\documentclass[11pt,reqno]{amsart}
\usepackage{amssymb,latexsym}
\usepackage{graphicx}
\usepackage{url}		%does nice formatting of URLs

\usepackage[utf8]{inputenc}
\usepackage[T1]{fontenc}
\usepackage{tabularx}
\usepackage{mathtools}
\usepackage{amsmath,amsfonts}
\usepackage{amsthm}

\calclayout

\newcommand\mymu[2][^n]{\prescript{#1\mkern-0.5mu}{}\mu_{#2}}

\makeatletter
\newcommand{\monthyear}[1]{%
  \def\@monthyear{\uppercase{#1}}}
\newcommand{\volnumber}[1]{%
  \def\@volnumber{\uppercase{#1}}}
\AtBeginDocument{%
\def\ps@plain{\ps@empty
  \def\@oddfoot{\@monthyear \hfil \thepage}%
  \def\@evenfoot{\thepage \hfil \@volnumber}}
\def\ps@firstpage{\ps@plain}
\def\ps@headings{\ps@empty
  \def\@evenhead{%
    \setTrue{runhead}%
    \def\thanks{\protect\thanks@warning}%
    \uppercase{}\hfil}%
  \def\@oddhead{%
    \setTrue{runhead}%
    \def\thanks{\protect\thanks@warning}%
    \hfill\uppercase{GENERALIZED SUMMATION}}%
  \let\@mkboth\markboth
  \def\@evenfoot{%
    \thepage \hfil \@volnumber}%
  \def\@oddfoot{%
    \@monthyear \hfil \thepage}%
  }%
\footskip=25pt
\pagestyle{headings}%
}
\makeatother

\theoremstyle{plain}
\numberwithin{equation}{section}
\newtheorem{thm}{Theorem}[section]

\newtheorem{lemma}[thm]{Lemma}

\newtheorem{definition}[thm]{Definition}

\begin{document}
%% replace the values in the next three lines by the correct information
\monthyear{}
\volnumber{}
\setcounter{page}{1}

\title{Generalized Nested Summation of Powers of Natural Numbers}
\author{Patibandla Chanakya*}
%\address{Department of Information Technology\\
%                IIIT Allahabad\\
%                Allahabad\\
%                Uttar Pradesh, India}
%\email{pis2017003@iiita.ac.in}
%\thanks{Research supported in part by the Natural Sciences and Engineering Research Council of Canada and by Emperor Frederick II of Sicily.}
\author{Putla Harsha}
%\address{Department of Information Technology\\
 %               IIIT Allahabad\\
%                Allahabad\\
%                Uttar Pradesh, India}
%\email{pis2017002@iiita.ac.in}

\begin{abstract}
In this paper, we provide a general framework for obtaining the formula for nested summation of powers of natural numbers. We define a special triangular array of numbers from which we can obtain the formula for nested summation of natural numbers at any particular power. Binomial coefficients play a key role in nesting. Our framework is very simple to understand.
\end{abstract}

\maketitle

\section{Introduction}

Formulas for the sum of powers of first $n$ natural numbers dates back to a long time ago. In his correspondence of 1636, Pierre de Fermat called the problem of finding formulas for sums of powers “what is perhaps the most beautiful problem of all arithmetic”\cite{pengelley2002bridge}. A general formula for the sum of powers of first $n$ natural numbers, $\sum\limits_{i=0}^{n} {i}^k$  was discovered using Bernoulli numbers and is called as Faulhaber's formula. Different authors extended and derived alternate proofs for Faulhaber's formula \cite{merca2015alternative}\cite{schumacher2016extended}. In this paper, we developed a formula for nested summation $\sum\limits^{(m)} n^k$, where $m$ stands for the number of times we apply summation.

Our framework is very simple to understand and can be used to generate the formula for nested summation involving the first $n$ natural numbers. The key part of the paper is the triangular array of numbers defined later.

We define and exemplify notation for nested summation in the second section along with some trivial notations that are used in this paper. We define a special triangular array of numbers (we call it Saras triangle) in the third section along with its formal definition. In the fourth section, we present a general formula and we provide few simple illustrations for intuitive understanding of the framework. We list out some basic lemmas in the fifth section.  In the sixth section, we provide a formula for the desired nested summation, which appeared in the fourth section, now along with proof. Finally, we end up with a conclusion.

\section{Notations Used}

We define the notation $\sum  n^k $ as $\sum  n^k = 
\sum\limits_{r=1}^{n} r^k = 1^k+2^k+3^k+\cdots+n^k$.  Throughout this paper, we denote set of whole numbers $\{0, 1, 2, 3, \cdots\}$ by $\mathbb{W}$ and set of natural numbers  $\{1, 2, 3, \cdots\}$ by $\mathbb{N}$. We are going to define the notation for nested summation $\sum\limits^{(m)} n^k$.

\begin{definition}
For $m \in \mathbb{W}, n\in \mathbb{N}, k \in \mathbb{W}$, we define $m-$ nested summation $\sum\limits^{(m)} n^k$ as

\[
\sum\limits^{(m)} n^k = 
     \begin{cases}
       n^k &\quad\text{if   } m = 0 \\\\
        \sum\limits^{(m-1)} 1^k+\sum\limits^{(m-1)} 2^k+\sum\limits^{(m-1)} 3^k+\cdots + \sum\limits^{(m-1)} n^k  &\quad\text{if } m \ge 1 \\ \\

     \end{cases}
\]

\end{definition}

We can denote $\sum\limits^{(m)} n^k$ as 

$$\sum\limits^{(m)} n^k = \underbrace{ \Big(  \sum \Big(  \sum  \Big(  \sum\cdots \Big( \sum}_{\text{m times}} n^k \Big) \Big) \Big)\Big) = \underbrace{\sum\sum\sum\cdots\sum}_{\text{m times}} n^k $$

Let us see the series for small values of $m$ for intuitive understanding

For $m=0$

$$\sum\limits^{(0)} n^k = n^k $$

For $m=1$

$$\sum\limits^{(1)} n^k =\sum n^k = 1^k +2^k +3^k +\cdots+n^k$$

For m =2: 
$$\sum\limits^{(2)} n^k =\sum\sum n^k =\sum 1^k + \sum 2^k + \sum 3^k +\cdots +\sum n^k $$

$$\implies \sum\limits^{(2)} n^k  = 1^k+(1^k +2^k) +(1^k +2^k +3^k) +\cdots +(1^k +2^k+3^k+\cdots +n^k)$$

For m =3: 

\[
\sum\limits^{(3)} n^k = 1^k+(1^k +(1^k + 2^k)) +(1^k + (1^k + 2^k) + (1^k + 2^k + 3^k)) +\cdots
\]
\[
+(1^k +(1^k+2^k)+ (1^k+2^k+3^k)+\cdots + (1^k + 2^k + 3^k + \cdots + n^k))
\]

Some properties related to nesting are discussed in section 5. We prove most of the statements in this paper using mathematical induction. We use $s$ for inducting variable. All other variables in the statement other than $s$ are considered to be arbitrary in the given scope unless stated.

\section{Saras Triangle}

Saras triangle is a special triangular array of numbers. We listed the first seven rows of Saras triangle. The numbering of rows starts with 0. In each row, the numbering of columns start with 0. $k^{th}$ row consists of $k+1$ numbers. The first number of any row is 1. Last number of $k^{th}$ row is $k!$. Other numbers are obtained by a simple linear combination of the two numbers above to it in the previous row. Let us see for example how the numbers in $4^{th}$ row $(k=4)$ are generated. First number is 1. $31 = 1(1) + 2(15), 180 = 2(15)+3(50), 390 = 3(50)+4(60), 360 = 4(60)+5(24), 120 = 5(24) = 120$.  All rows can be generated using the same procedure. This triangle has several interesting properties.

\begin{table}[h]
\centering
\begin{tabular}{>{$k=}l<{$\hspace{12pt}}*{13}{c}}
\hline
0 &&&&&&&1&&&&&&\\
\hline
1 &&&&&&1&&1&&&&&\\
\hline
2 &&&&&1&&3&&2&&&&\\
\hline
3 &&&&1&&7&&12&&6&&&\\
\hline
4 &&&1&&15&&50&&60&&24&&\\
\hline
5 &&1&&31&&180&&390&&360&&120&\\
\hline
6 &1&&63&&602&&2100&&3360&&2520&&720\\
\hline
\end{tabular}
\caption{Saras triangle with first seven rows}
\end{table}

In order to construct the formula for nested summation of power $k$ of natural numbers, only numbers from $k^{th}$ row of Saras triangle are used along with binomial coefficients.

Let us denote the entry in $k^{th}$ row  and the $r^{th}$ column in Saras triangle by $\mymu[k]{r}$ and can be defined as follows:

For $k \in \mathbb{W}$ and $r \in \mathbb{W}$
\begin{equation}
\mymu[k]{r}
 = 
     \begin{cases}
       1 &\quad\text{if   } r = 0 \\
        r  \Big(\mymu[k-1]{r-1}\Big)+ (r + 1) \Big(\mymu[k-1]{r} \Big)  &\quad\text{if }1 \le r \le k\\ 
        0 &\quad\text{if   } r > k  \\
     \end{cases}
\end{equation}

In the next section, we present and exemplify the formula for nested summation using the entries from Saras triangle and binomial coefficients.

\section{Formula and Few Illustrations }

For $m \in \mathbb{W}, n \in \mathbb{N}, k \in \mathbb{W}$, the formula for $m-$nested summation of power $k$ of first $n$ natural numbers is given by

$$\sum\limits^{(m)} n^k = \sum\limits_{i=0}^{k} \binom{n+m-1}{m+i}  \mymu[k]{i}  $$

Let us generate formulas for $\sum n^k$ for $k=1$ to 4 from Saras triangle using the above formula. 

For $m=1, k=1$

\begin{equation*}
\sum\limits^{(1)} n^1 = \sum n = \sum\limits_{i=0}^{1} \binom{n}{1+i} \Big(\mymu[1]{i} \Big) =  \binom{n}{1} (1) + \binom{n}{2} (1) = \dfrac{n(n+1)}{2}
\end{equation*}

For $m=1, k=2$

\begin{equation*}
\sum\limits^{(1)} n^2 = \sum n^2 = \sum\limits_{i=0}^{2} \binom{n}{1+i} \Big(\mymu[2]{i} \Big) =  \binom{n}{1} (1) + \binom{n}{2} (3) + \binom{n}{3} (2)
 = \dfrac{n(n+1)(2n+1)}{6}
\end{equation*}

For $m=1, k=3$

\begin{equation*}
\sum\limits^{(1)} n^3 = \sum n^3 = \sum\limits_{i=0}^{3} \binom{n}{1+i} \Big(\mymu[3]{i} \Big)=  \binom{n}{1} (1) + \binom{n}{2} (7) + \binom{n}{3} (12) + \binom{n}{4} (6)
 = \dfrac{n^2(n+1)^2}{4}
\end{equation*}

For $m=1, k=4$
\begin{equation*}
\sum\limits^{(1)} n^4 = \sum n^4 = \sum\limits_{i=0}^{4} \binom{n}{1+i} \Big(\mymu[4]{i} \Big) =  \binom{n}{1} (1) + \binom{n}{2} (15) + \binom{n}{3} (50) + \binom{n}{4} (60) +\binom{n}{5} (24) \end{equation*}

\begin{equation*}
= \dfrac{n(n+1)(2n+1)(3n^2+3n-1)}{30}
\end{equation*}

For $m=2$ and $k=5$

\begin{equation*}
\sum\limits^{(2)} n^5 = 1^5 + (1^5+2^5) +\cdots + (1^5+2^5+3^5+\cdots+n^5) 
\end{equation*}

\begin{equation*}
= \binom{n+1}{2} + 31 \binom{n+1}{3}  + 180\binom{n+1}{4} + 390 \binom{n+1}{5} + 360 \binom{n+1}{6} + 120 \binom{n+1}{7} 
\end{equation*}

We can observe that the entries of row $k$ is taken for nested summation of first $n$ natural numbers of power $k$. Based on nesting, the binomial coefficients changes. Thus Saras triangle gives the desired numbers for any power. The formula is concise and Saras triangle is sufficient enough to generate nested summation of any given power. Till now, we completely understood the framework for generating the formula for nested summation of powers of first $n$ natural numbers. In the next two sections, we prove some lemmas and prove the formula for nested summation using those lemmas.

\section{Lemmas}

We denote the $m$-nested summation of  $k^{th}$ power of first $n$ natural numbers by $\sum\limits^{(m)} n^k$. The following trivial properties we listed in this section will be useful while proving our theorem in section 6.

\begin{lemma} For $m \in \mathbb{W}, n \in \mathbb{N}, k \in \mathbb{W}$
\begin{equation*}
\sum\limits^{(m+1)} n^k = \sum\limits^{(m)} 1^k + \sum\limits^{(m)} 2^k  +\sum\limits^{(m)} 3^k +\cdots+\sum\limits^{(m)} n^k 
\end{equation*}
\end{lemma}
 
\begin{proof}

It is directly obtained from definition 2.1.
\end{proof}

\begin{lemma} For $m \in \mathbb{W}, n \in \mathbb{N}-\{1\}, k \in \mathbb{W}$
\begin{equation*}
 \sum\limits^{(m)} n^{k+1} = \Big( n \sum\limits^{(m)} n^{k} \Big)  - \Big( m \sum \limits^{(m+1)} (n-1)^k \Big) 
\end{equation*}
\end{lemma}

\begin{proof}

We know that, the coefficient of $r^k$ in $\sum\limits^{(m)} n^{k}$ is $\binom{n-r+m-1}{m-1}$ for $1\le r \le n$ 

Now, 

\[
\sum\limits^{(m)} n^{k} = \sum \limits_{r=1}^{n} \binom{n-r+m-1}{m-1} r^k
\]

and

\[
\sum\limits^{(m)} n^{k+1} = \sum \limits_{r=1}^{n}  \binom{n-r+m-1}{m-1} r^{k+1} 
\]

\begin{equation*}
\implies  n \sum\limits^{(m)} n^{k} -  \sum\limits^{(m)} n^{k+1} =  \sum \limits_{r=1}^{n} \binom{n-r+m-1}{m-1} (n-r) r^k  =  \sum \limits_{r=1}^{n-1} \binom{n-r+m-1}{m} (m) r^k 
\end{equation*}

\[
\implies  n \sum\limits^{(m)} n^{k} -  \sum\limits^{(m)} n^{k+1} = m \sum \limits_{r=1}^{n-1} \binom{(n-1)-r+(m+1)-1}{(m+1)-1} r^k  =  m \sum\limits^{(m+1)} (n-1)^{k} 
\]

\[
\therefore \sum\limits^{(m)} n^{k+1} =    n \sum\limits^{(m)} n^{k} -  m \sum\limits^{(m+1)} (n-1)^{k} 
\]

\end{proof}

\begin{lemma} For $m \in \mathbb{W}, n \in \mathbb{N}, k \in \mathbb{W}$
\begin{equation*}
\sum\limits^{(m)} (n+1)^k = \Big(\sum\limits^{(m)} n^k + \sum\limits^{(m-1)} n^k  +\sum\limits^{(m-2)} n^k +\cdots+\sum\limits^{(1)} n^k \Big) +(n+1)^k 
\end{equation*}
\end{lemma}

\begin{proof}
\begin{equation*}
\sum\limits^{(m)} (n+1)^k = \sum\limits^{(m-1)} 1^k+\sum\limits^{(m-1)} 2^k+\sum\limits^{(m-1)} 3^k+\cdots + \sum\limits^{(m-1)} n^k  + \sum\limits^{(m-1)} (n+1)^k 
\end{equation*}

We obtained the following step from the above step by using definition 2.1

\begin{equation*}
\implies \sum\limits^{(m)} (n+1)^k = \sum\limits^{(m)} n^k + \sum\limits^{(m-1)} (n+1)^k  = \sum\limits^{(m)} n^k + \sum\limits^{(m-1)} n^k +  \sum\limits^{(m-2)} (n+1)^k
\end{equation*}
If we continue expanding till zeroth nesting, we get
\begin{equation*}
\sum\limits^{(m)} (n+1)^k =  \sum\limits^{(m)} n^k + \sum\limits^{(m-1)} n^k +  \sum\limits^{(m-2)} n^k +\cdots + \sum\limits^{(1)} n^k + \sum\limits^{(0)}(n+1)^k
\end{equation*}

\begin{equation*}
\therefore \sum\limits^{(m)} (n+1)^k =  \Big(\sum\limits^{(m)} n^k + \sum\limits^{(m-1)} n^k +  \sum\limits^{(m-2)} n^k +\cdots + \sum\limits^{(1)} n^k \Big)+ (n+1)^k
\end{equation*}
\end{proof}

\begin{lemma} For $k \in \mathbb{W}$ and $r \in \mathbb{W}$ 
\[
\mymu[k]{r} = 
     \begin{cases}
       1 &\quad\text{if   } r = 0 \\\\
        \sum \limits_{i=r-1}^{k-1} \binom{k}{i} \mymu[i]{r-1}  &\quad\text{if }1 \le r \le k \\ \\
        0 &\quad\text{if  }  r>k\\
     \end{cases}
\]
\end{lemma}

\begin{proof}

We prove this statement using mathematical induction. We denote the given statement by $P(k,r)$. That is 

\begin{equation*}
P(k,r) : \mymu[k]{r} = 
     \begin{cases}
       1 &\quad\text{if   } r = 0 \\\\
        \sum \limits_{i=r-1}^{k-1} \binom{k}{i} \mymu[i]{r-1}  &\quad\text{if }1 \le r \le k \\ \\
        0 &\quad\text{if  }  r>k\\
     \end{cases}
\end{equation*}

\underline{Basis Case}: We prove $P(0, r), P(1,r), P(2,r), P(k,0), P(k,1)$

\[
P(0,r) : \mymu[0]{r} = 
     \begin{cases}
       1 &\quad\text{if   } r = 0 \\ 
       0 &\quad\text{if   } r > 0 \\
       \end{cases}
\]

\[
P(1,r) : \mymu[1]{r} = 
     \begin{cases}
       1 &\quad\text{if   } r = 0,1 \\ 
       0 &\quad\text{if   } r > 1 \\
       \end{cases}
\]

\[
P(2,r) : \mymu[2]{r} = 
     \begin{cases}
       1 &\quad\text{if   } r = 0 \\ 
       3 &\quad\text{if   } r = 1 \\
       2 &\quad\text{if   } r = 2 \\
       0 &\quad\text{if   } r > 2 \\
       \end{cases}
\]

\[
P(k,0) : \mymu[k]{0} =  1 
\]

\[
P(k,1) : \mymu[k]{1} = 
     \begin{cases}
       0 &\quad\text{if   } k = 0 \\ 
       2^k - 1 &\quad\text{if   } k >= 1 \\
       \end{cases}
\]

All the statements $P(0,r), P(1,r), P(2,r), P(k,0), P(k,1)$ are consistent with the definition (1). Hence true.

\underline{Induction Step:}  In this step we prove $ \Big[ P(s,r-1) \land P(s,r) \Big] \implies P(s+1, r)$ for all $s \ge 2$ and $2 \le r \le k$
\begin{equation*}
\mymu[s+1]{r} =  r \big (\mymu[s]{r-1} \big) +(r+1) \big( \mymu[s]{r} \big) = r \Bigg(\sum \limits_{i=r-2}^{s-1} \binom{s}{i} \mymu[i]{r-2} \Bigg)+(r+1) \Bigg(\sum \limits_{i=r-1}^{s-1} \binom{s}{i} \mymu[i]{r-1} \Bigg)
\end{equation*}

\begin{equation*}
\implies \mymu[s+1]{r} =   r \Bigg( \sum \limits_{i=r-1}^{s-1} \binom{s}{i} \mymu[i]{r-2} + \binom{s}{r-2} \mymu[r-2]{r-2}  \Bigg)+(r+1) \Bigg(\sum \limits_{i=r-1}^{s-1} \binom{s}{i} \mymu[i]{r-1} \Bigg)
\end{equation*}

\begin{equation*}
\implies \mymu[s+1]{r} =   r \Bigg(\binom{s}{r-2} \mymu[r-2]{r-2} \Bigg) + r \Bigg(\sum \limits_{i=r-1}^{s-1} \binom{s}{i} \mymu[i]{r-2} \Bigg)+(r+1) \Bigg(\sum \limits_{i=r-1}^{s-1} \binom{s}{i} \mymu[i]{r-1} \Bigg)
\end{equation*}

{\scriptsize 
\begin{equation*}
\implies \mymu[s+1]{r} =   r \Bigg(\binom{s}{r-2} \mymu[r-2]{r-2} \Bigg) + (r-1) \Bigg(\sum \limits_{i=r-1}^{s-1} \binom{s}{i} \mymu[i]{r-2} \Bigg)+(r) \Bigg(\sum \limits_{i=r-1}^{s-1} \binom{s}{i} \mymu[i]{r-1} \Bigg) +  \Bigg(\sum \limits_{i=r-1}^{s-1} \binom{s}{i} \mymu[i]{r-2} \Bigg)+ \Bigg(\sum \limits_{i=r-1}^{s-1} \binom{s}{i} \mymu[i]{r-1} \Bigg)
\end{equation*}

}

\begin{equation*}
\implies \mymu[s+1]{r} =   r \Bigg(\binom{s}{r-2} \mymu[r-2]{r-2} \Bigg) + \Bigg(\sum \limits_{i=r-1}^{s-1} \binom{s}{i} \mymu[i+1]{r-1} \Bigg) +  \Bigg(\sum \limits_{i=r-1}^{s-1} \binom{s}{i} \mymu[i]{r-2} \Bigg)+ \Bigg(\sum \limits_{i=r-1}^{s-1} \binom{s}{i} \mymu[i]{r-1} \Bigg)
\end{equation*}

\begin{equation*}
\implies \mymu[s+1]{r} =   r \Bigg(\binom{s}{r-2} \mymu[r-2]{r-2} \Bigg) + \Bigg(\sum \limits_{i=r}^{s} \binom{s}{i-1} \mymu[i]{r-1} \Bigg) +  \Bigg(\sum \limits_{i=r-1}^{s-1} \binom{s}{i} \mymu[i]{r-2} \Bigg)+ \Bigg(\sum \limits_{i=r-1}^{s-1} \binom{s}{i} \mymu[i]{r-1} \Bigg)
\end{equation*}

\begin{equation*}
\implies \mymu[s+1]{r} =   r \Bigg(\binom{s}{r-2} \mymu[r-2]{r-2} \Bigg) +  \Bigg(\sum \limits_{i=r}^{s-1} \binom{s}{i-1} \mymu[i]{r-1} \Bigg) + \Bigg(\binom{s}{s-1} \mymu[s]{r-1} \Bigg)+ 
\end{equation*}

\begin{equation*}
\Bigg(\sum \limits_{i=r-1}^{s-1} \binom{s}{i} \mymu[i]{r-2} \Bigg)+ \Bigg(\sum \limits_{i=r}^{s-1} \binom{s}{i} \mymu[i]{r-1} \Bigg) + \Bigg( \binom{s}{r-1} \mymu[r-1]{r-1} \Bigg)
\end{equation*}

\begin{equation*}
\implies \mymu[s+1]{r} =   r \Bigg(\binom{s}{r-2} \mymu[r-2]{r-2} \Bigg) +  \Bigg(\sum \limits_{i=r}^{s-1} \binom{s+1}{i} \mymu[i]{r-1} \Bigg) + \Bigg(\binom{s}{s-1} \mymu[s]{r-1} \Bigg)+ 
\end{equation*}

\begin{equation*}
\Bigg(\sum \limits_{i=r-1}^{s-1} \binom{s}{i} \mymu[i]{r-2} \Bigg) + \Bigg( \binom{s}{r-1} \mymu[r-1]{r-1} \Bigg)
\end{equation*}

\begin{equation*}
\implies \mymu[s+1]{r} =   (r-1) \Bigg(\binom{s}{r-2} \mymu[r-2]{r-2} \Bigg) +  \Bigg(\sum \limits_{i=r}^{s-1} \binom{s+1}{i} \mymu[i]{r-1} \Bigg) + \Bigg(\binom{s}{s-1} \mymu[s]{r-1} \Bigg)+
\end{equation*}

\begin{equation*}
 \Bigg(\binom{s}{r-2} \mymu[r-2]{r-2} \Bigg) + \Bigg(\sum \limits_{i=r-1}^{s-1} \binom{s}{i} \mymu[i]{r-2} \Bigg) + \Bigg( \binom{s}{r-1} \mymu[r-1]{r-1} \Bigg)
\end{equation*}

{\scriptsize 
\begin{equation*}
\implies \mymu[s+1]{r} =   (r-1) \Bigg(\binom{s}{r-2} \mymu[r-2]{r-2} \Bigg) +  \Bigg(\sum \limits_{i=r}^{s-1} \binom{s+1}{i} \mymu[i]{r-1} \Bigg) + \Bigg(\binom{s}{s-1} \mymu[s]{r-1} \Bigg) + \Bigg(\sum \limits_{i=r-2}^{s-1} \binom{s}{i} \mymu[i]{r-2} \Bigg) + \Bigg( \binom{s}{r-1} \mymu[r-1]{r-1} \Bigg)
\end{equation*}
}

\begin{equation*}
\implies \mymu[s+1]{r} =   \Bigg(\binom{s}{r-2} \mymu[r-1]{r-1} \Bigg) +  \Bigg(\sum \limits_{i=r}^{s-1} \binom{s+1}{i} \mymu[i]{r-1} \Bigg) + \Bigg(\binom{s}{1} \mymu[s]{r-1} \Bigg) + \Big( \mymu[s]{r-1} \Big) + \Bigg( \binom{s}{r-1} \mymu[r-1]{r-1} \Bigg)
\end{equation*}

\begin{equation*}
\implies \mymu[s+1]{r} =   \Bigg(\binom{s}{r-2} \mymu[r-1]{r-1} \Bigg) +  \Bigg(\sum \limits_{i=r}^{s-1} \binom{s+1}{i} \mymu[i]{r-1} \Bigg) + \Bigg(\binom{s+1}{1} \mymu[s]{r-1} \Bigg)  + \Bigg( \binom{s}{r-1} \mymu[r-1]{r-1} \Bigg)
\end{equation*}

\begin{equation*}
\implies \mymu[s+1]{r} =   \Bigg(\binom{s}{r-2} \mymu[r-1]{r-1} \Bigg) +  \Bigg(\sum \limits_{i=r}^{s-1} \binom{s+1}{i} \mymu[i]{r-1} \Bigg) + \Bigg(\binom{s+1}{s} \mymu[s]{r-1} \Bigg)  + \Bigg( \binom{s}{r-1} \mymu[r-1]{r-1} \Bigg)
\end{equation*}

\begin{equation*}
\implies \mymu[s+1]{r} =   \Bigg(\binom{s}{r-2} \mymu[r-1]{r-1} \Bigg) +  \Bigg(\sum \limits_{i=r}^{s} \binom{s+1}{i} \mymu[i]{r-1} \Bigg) + \Bigg( \binom{s}{r-1} \mymu[r-1]{r-1} \Bigg)
\end{equation*}

\begin{equation*}
\implies \mymu[s+1]{r} =   \Bigg(\sum \limits_{i=r}^{s} \binom{s+1}{i} \mymu[i]{r-1} \Bigg) + \Bigg( \binom{s+1}{r-1} \mymu[r-1]{r-1} \Bigg)
\end{equation*}

\begin{equation*}
\implies \mymu[s+1]{r} =   \sum \limits_{i=r-1}^{s} \binom{s+1}{i} \mymu[i]{r-1} 
\end{equation*}

\begin{equation*}
\implies \mymu[s+1]{r} =  \sum \limits_{i=r-1}^{(s+1)-1} \binom{s+1}{i} \mymu[i]{r-1} 
\end{equation*}

Hence the statement is proved.
\end{proof}

\begin{lemma} For $n \in  \mathbb{W}, k \in \mathbb{W}$
\begin{equation*}
(n+1)^k = \sum\limits_{i=0}^{k} \binom{n}{i} \Big( \mymu[k]{i} \Big)
\end{equation*}
\end{lemma}

\begin{proof}
We prove this lemma using mathematical induction. Let us denote the statement by $P(n,k)$. That is 

\[
P(n,k) : (n+1)^k = \sum\limits_{i=0}^{k} \binom{n}{i}  \Big( \mymu[k]{i} \Big)
\]

\underline{Basis step:} We prove for $P(0,0)$

\[
LHS = (n+1)^k = (0+1)^0 = 1
\]
\[
RHS =  \sum\limits_{i=0}^{k} \binom{n}{i}  \Big( \mymu[k]{i} \Big) = \sum\limits_{i=0}^{0} \binom{0}{i} \mymu[0]{i} = \binom{0}{0} \mymu[0]{0} = 1
\]

Hence $P(0,0)$ is true.

\underline{Induction steps:} We divide this step into two substeps. 

\underline{Step 1} : First step is that we assume $P(n,k)$  is true for $k=s$. That is 

\[
 (n+1)^s = \sum\limits_{i=0}^{s} \binom{n}{i}  \Big( \mymu[s]{i} \Big)
\]
Now we prove $ P(n,s+1)$. That is $P(n,s) \implies P(n,s+1)$

\[
(n+1)^{s+1} = (n+1)(n+1)^s = (n+1) \Bigg( \sum\limits_{i=0}^{s} \binom{n}{i} \mymu[s]{i} \Bigg) 
\]

\[
= \sum\limits_{i=0}^{s} \Bigg( \Big( \dfrac{n+1}{i+1} \Big) \binom{n}{i} (i+1) \big( \mymu[s]{i} \big) \Bigg) = \sum\limits_{i=0}^{s}  \Bigg( (i+1) \binom{n+1}{i+1} \mymu[s]{i} \Bigg)
\]

\[
\implies (n+1)^{s+1} = \sum\limits_{i=0}^{s} \Bigg( \binom{n}{i} (i+1) (\mymu[s]{i})  + \binom{n}{i+1} (i+1) \big(\mymu[s]{i} \big) \Bigg)
\]

\[
=  \sum\limits_{i=0}^{s}  \Bigg( \binom{n}{i} (i+1) \big(\mymu[s]{i} \big) \Bigg) +  \sum\limits_{i=1}^{s+1}  \Bigg( \binom{n}{i} (i) \big(\mymu[s]{i-1} \big) \Bigg)
\]

\[
\implies (n+1)^{s+1} = \binom{n}{0} (1) \big( \mymu[s]{0} \big) + \Bigg( \sum\limits_{i=1}^{s}  \binom{n}{i} (i+1) \big(\mymu[s]{i} \big) \Bigg) + \Bigg( \sum\limits_{i=1}^{s}  \binom{n}{i} (i) \big( \mymu[s]{i-1} \big) \Bigg)+ \binom{n}{s+1}(s+1)  \big( \mymu[s]{s} \big)
\]

\[
\implies (n+1)^{s+1} = \binom{n}{0}  \mymu[s]{0}  + \Bigg( \sum\limits_{i=1}^{s}  \binom{n}{i} \Big( (i) \big( \mymu[s]{i-1} \big) + (i+1) \big( \mymu[s]{i} \big) \Big) \Bigg) + \binom{n}{s+1} (s+1) \big( \mymu[s]{s} \big)
\]

\[
\implies (n+1)^{s+1} = \binom{n}{0} \mymu[s]{0} + \Bigg( \sum\limits_{i=1}^{s}   \binom{n}{i} \mymu[s+1]{i} \Bigg) + \binom{n}{s+1} (s+1) \big( \mymu[s]{s} \big)
\]

\[
\implies (n+1)^{s+1} = \binom{n}{0} \mymu[s+1]{0} + \Bigg( \sum\limits_{i=1}^{s}   \binom{n}{i} \mymu[s+1]{i} \Bigg) + \binom{n}{s+1} \mymu[s+1]{s+1}
\]

\[
\implies (n+1)^{s+1} = \sum\limits_{i=0}^{s+1}   \binom{n}{i} \mymu[s+1]{i}
\]

\underline{Step 2} : We prove $\Big[P(s,0) \land P(s,1) \land \cdots \land P(s,k) \Big] \implies P(s+1,k)$

\[
  (s+2)^{k} =  \sum \limits_{i=0}^{k} \binom{k}{i} (s+1)^i = \sum \limits_{i=0}^{k} \Bigg( \binom{k}{i}  \sum\limits_{j=0}^{i} \binom{s}{j} \mymu[i]{j} \Bigg) =   \Bigg(\sum \limits_{i=0}^{k-1} \binom{s}{i} \Big( \sum\limits_{j=i}^{k-1} \binom{k}{j} \mymu[j]{i} \Big) \Bigg) + \binom{k}{k} \Bigg( \sum\limits_{j=0}^{k} \binom{s}{j} \mymu[k]{j} \Bigg) 
\]

The next step is obtained by applying Lemma 5.4 on previous step
\[
\implies  (s+2)^{k} =  \Bigg( \sum\limits_{i=0}^{k} \binom{s}{i} \mymu[k]{i} \Bigg)  + \Bigg( \sum \limits_{i=0}^{k-1} \binom{s}{i} \mymu[k]{i+1} \Bigg) =    \binom{s}{0} \mymu[k]{0} + \Bigg( \sum\limits_{i=0}^{k-1} \binom{s}{i+1} \mymu[k]{i+1} \Bigg)  + \Bigg( \sum \limits_{i=0}^{k-1} \binom{s}{i} \mymu[k]{i+1} \Bigg)
\]

\[
\implies  (s+2)^{k} =  \binom{s}{0} \mymu[k]{0} + \sum\limits_{i=0}^{k-1} \binom{s+1}{i+1} \mymu[k]{i+1}  =  \binom{s+1}{0} \mymu[k+1]{0} +  \sum\limits_{i=0}^{k-1} \binom{s+1}{i+1} \mymu[k]{i+1}  = \sum\limits_{i=0}^{k} \binom{s+1}{i} \mymu[k]{i} 
\]

\[
\implies (s+2)^{k} = \sum\limits_{i=0}^{k} \binom{s+1}{i} \mymu[k]{i}
\]

Hence the statement is proved.
%$\implies (s+1)^{k+1} = \sum\limits_{i=0}^{k}  \binom{s+1}{i+1} (i)\mymu[k]{i}  + \binom{s+1}{i+1} \mymu[k]{i} =  \sum\limits_{i=1}^{k}  \binom{s+1}{i+1} (i)\mymu[k]{i}  +  \sum\limits_{i=0}^{k}  \binom{s+1}{i+1} \mymu[k]{i} $

\end{proof}

\section{Formula}
For $m \in \mathbb{W}, n \in \mathbb{N}, k \in \mathbb{W}$ , the $m-$nested summation of power $k$ of first $n$ natural numbers is given by

$$\sum\limits^{(m)} n^k = \sum\limits_{i=0}^{k} \binom{n+m-1}{m+i} 
\mymu[k]{i} $$

\begin{proof}

We prove the above statement by using mathematical induction. We denote the statement as $P(m,n,k)$. That is

$$P(m,n,k) : \sum\limits^{(m)} n^k = \sum\limits_{i=0}^{k} \binom{n+m-1}{m+i} \mymu[k]{i}$$

\textbf{Basis case} : 

In this case we prove $P(0,1,0)$ is true

$$LHS = \sum\limits^{(0)} 1^0 = 1^0 = 1$$

$$RHS = \sum\limits_{i=0}^{0} \binom{1+0-1}{0+i} 
\mymu[0]{i} = \sum\limits_{i=0}^{0} \binom{0}{0+i} 
\mymu[0]{i} = \binom{0}{0} \mymu[0]{0} = 1$$

Hence $P(0,1,0)$ is proved.

\textbf{Inductive step} : 

We divide this step in to the following 3 parts:

$$\Big[P(s,1,k) \land P(s,2,k) \land P(s,3,k)  \land \cdots  \land  P(s,n,k)\Big] \implies P(s+1,n,k)$$

$$\Big[P(m,n,s) \land P(m+1,n-1,s)\Big] \implies  P(m,n,s+1)$$

$$\Big[P(m,s,k) \land P(m-1,s,k) \land P(m-2,s,k) \land \cdots \land P(0,s,k)\Big] \implies P(m,s+1,k)$$

We prove each of the above implication below.

\underline{Part 1:} $\Big[P(s,1,k) \land P(s,2,k) \land P(s,3,k) \land \cdots \land P(s,n,k)\Big] \implies  P(s+1,n,k)$

Assume that the statement is true for $m=s$. That is

$$\sum\limits^{(s)} n^k = \sum\limits_{i=0}^{k} \binom{n+s-1}{s+i} \LARGE
\mymu[k]{i} $$

Now, we have to prove for $m = s+1$

From Lemma 5.1, we know that
\[
\sum\limits^{(s+1)} n^k = \sum \Big(\sum\limits^{(s)} n^k\Big) =  \sum\limits^{(s)} 1^k + \sum\limits^{(s)} 2^k + \sum\limits^{(s)} 3^k+ \cdot + \sum\limits^{(s)} n^k \]                                                                                                                                                                                                                                                                                                                                                                                                                                                                                                                                                                                                                                                                                                                                                                                                                                                                                                                                                                                                                                                                                                                                                                                                                                                                                                                                                                                                                                                                                                                                                                                                                                                                                                                                     

{\scriptsize
\[ 
\implies \sum\limits^{(s+1)} n^k = \sum\limits_{i=0}^{k} \binom{1+s-1}{s+i} \LARGE \mymu[k]{i} + \sum\limits_{i=0}^{k} \binom{2+s-1}{s+i} \LARGE \mymu[k]{i} + \sum\limits_{i=0}^{k} \binom{3+s-1}{s+i} \LARGE \mymu[k]{i} +\cdots + \sum\limits_{i=0}^{k} \binom{n+s-1}{s+i} \LARGE \mymu[k]{i}
\]
}

\[ 
\implies \sum\limits^{(s+1)} n^k =  \sum\limits_{i=0}^{k} \binom{s}{s+i} \LARGE \mymu[k]{i} + \sum\limits_{i=0}^{k} \binom{s+1}{s+i} \LARGE \mymu[k]{i} + \sum\limits_{i=0}^{k} \binom{s+2}{s+i} \LARGE \mymu[k]{i} +\cdots + \sum\limits_{i=0}^{k} \binom{n+s-1}{s+i} \LARGE \mymu[k]{i}
\]

\[
\implies \sum\limits^{(s+1)} n^k =\mymu[k]{0} \sum\limits_{r=0}^{n-1} \binom{s+r}{s} + \mymu[k]{1} \sum\limits_{r=0}^{n-1} \binom{s+r}{s+1} + \mymu[k]{2} \sum\limits_{r=0}^{n-1} \binom{s+r}{s+2} +\cdots+\mymu[k]{k} \sum\limits_{r=0}^{n-1} \binom{s+r}{s+k}
\]

\[
\implies \sum\limits^{(s+1)} n^k =\mymu[k]{0} \binom{s+n}{s+1} + \mymu[k]{1} \binom{s+n}{s+2}  + \mymu[k]{2} \binom{s+n}{s+3} 
+\cdots+\mymu[k]{k} \binom{s+n}{s+k+1} 
\]

\[
=\sum\limits_{i=0}^{k}  \mymu[k]{i} \binom{s+n}{s+1+i}
\]

\[
\implies \sum\limits^{(s+1)} n^k  =\sum\limits_{i=0}^{k}  \mymu[k]{i} \binom{n+(s+1)-1}{(s+1)+i} 
\]

\underline{Part 2:}$\Big[P(m,n,s) \land P(m+1,n-1,s)\Big] \implies P(m,n,s+1)$

Assume that the statement is true for $k=s$. That is

$$\sum\limits^{(m)} n^s = \sum\limits_{i=0}^{s} \binom{n+m-1}{m+i} 
\mymu[s]{i} $$

Now, we have to prove for $k = s+1$

From Lemma 5.2, we know that

\[
\sum\limits^{(m)} n^{s+1} = n \Big(\sum\limits^{(m)} n^{s}  \Big) - m \Big( \sum\limits^{(m+1)} (n-1)^{s} \Big) 
\]

\[
 \implies \sum\limits^{(m)} n^{s+1} = n \Big(\sum\limits_{i=0}^{s} \binom{n+m-1}{m+i} 
\mymu[s]{i} \Big) - m \Big(\sum\limits_{i=0}^{s} \binom{n+m-1}{m+1+i} 
\mymu[s]{i} \Big) 
\]

\[
 \implies \sum\limits^{(m)} n^{s+1}=  \Big(\sum\limits_{i=0}^{s} n\binom{n+m-1}{m+i} 
\mymu[s]{i} \Big) -  \Big(\sum\limits_{i=0}^{s} m\binom{n+m-1}{m+1+i} 
\mymu[s]{i} \Big) 
\]

\[
\implies \sum\limits^{(m)} n^{s+1} =  \sum\limits_{i=0}^{s}   \Big( \mymu[s]{i} \Big) \Bigg( n \binom{n+m-1}{m+i} - m \binom{n+m-1}{m+1+i} \Bigg) 
 \]

\[
\implies \sum\limits^{(m)} n^{s+1} =  \sum\limits_{i=0}^{s} \Big( \mymu[s]{i} \Big)  (i+1) \binom{n+m}{m+i+1}  
 \]

\[
\implies \sum\limits^{(m)} n^{s+1} =  \sum\limits_{i=0}^{s} (i+1) \Big( \mymu[s]{i} \Big) \binom{n+m-1}{m+i} + \sum\limits_{i=0}^{s} (i+1) \Big( \mymu[s]{i} \Big) \binom{n+m-1}{m+i+1} 
 \]

\[
\implies \sum\limits^{(m)} n^{s+1} =  \mymu[s]{0}  \binom{n+m-1}{m} + \sum\limits_{i=1}^{s} (i+1)  \Big( \mymu[s]{i} \Big) \binom{n+m-1}{m+i} 
\]

\[
+ \sum\limits_{i=0}^{s-1} (i+1)  \Big( \mymu[s]{i} \Big)  \binom{n+m-1}{m+i+1} + (s+1) \Big( \mymu[s]{s} \Big) \binom{n+m-1}{m+s+1}
 \]

\[
\implies \sum\limits^{(m)} n^{s+1} =  \mymu[s+1]{0}  \binom{n+m-1}{m} + \sum\limits_{i=1}^{s} (i+1)  \Big(\mymu[s]{i}\Big)  \binom{n+m-1}{m+i} 
\]

\[
+ \sum\limits_{i=1}^{s} i  \Big( \mymu[s]{i-1} \Big) \binom{n+m-1}{m+i} + \Big( \mymu[s+1]{s+1} \Big)  \binom{n+m-1}{m+s+1}
 \]

\[
\implies \sum\limits^{(m)} n^{s+1} =  \mymu[s+1]{0}  \binom{n+m-1}{m} + \sum\limits_{i=1}^{s} \Big( i \big( \mymu[s]{i-1} \big) + (i+1) \big( \mymu[s]{i} \big) \Big)  \binom{n+m-1}{m+i} + \mymu[s+1]{s+1}  \binom{n+m-1}{m+s+1}
 \]

\[
\implies \sum\limits^{(m)} n^{s+1} =  \mymu[s+1]{0}  \binom{n+m-1}{m} + \sum\limits_{i=1}^{s} \mymu[s+1]{i}   \binom{n+m-1}{m+i} + \mymu[s+1]{s+1}  \binom{n+m-1}{m+s+1}
 \]

\[
\implies \sum\limits^{(m)} n^{s+1} = \sum\limits_{i=0}^{s+1}   \binom{n+m-1}{m+i} \mymu[s+1]{i}  
 \]

\underline{Part 3:} $\Big[P(m,s,k) \land P(m-1,s,k) \land P(m-2,s,k) \land \cdots \land P(0,s,k)\Big]  \implies  P(m,s+1,k)$

Assume that the statement is true for $n=s$. That is

\[
\sum\limits^{(m)} s^k = \sum\limits_{i=0}^{k} \binom{s+m-1}{m+i} 
\mymu[k]{i}
\]

Now, we have to prove for $n = s+1$

From Lemma 5.3, we know that

\[
\sum\limits^{(m)} (s+1)^k = \Bigg(\sum\limits^{(m)} s^k + \sum\limits^{(m-1)} s^k + \sum\limits^{(m-2)} s^k +  \cdots + \sum\limits^{(1)} s^k\Bigg) +  (s+1)^k
\]

After using lemma 5.5 on the above step, we get

{\scriptsize

\[
\implies \sum\limits^{(m)} (s+1)^k  = \Bigg( \sum\limits_{i=0}^{k} \binom{s+m-1}{m+i} 
\mymu[k]{i} + \sum\limits_{i=0}^{k} \binom{s+m-2}{m+i-1} 
\mymu[k]{i} + \sum\limits_{i=0}^{k} \binom{s+m-3}{m+i-2} 
\mymu[k]{i} +\cdots + \sum\limits_{i=0}^{k} \binom{s}{1+i} 
\mymu[k]{i} \Bigg) +  \sum\limits_{i=0}^{k} \binom{s}{i} 
\mymu[k]{i}
\]

}

\[
\implies \sum\limits^{(m)} (s+1)^k = \sum\limits_{i=0}^{k} \Bigg(  \binom{s+m-1}{m+i} + \binom{s+m-2}{m+i-1}  + \binom{s+m-3}{m+i-2}  + \cdots + \binom{s}{1+i} + \binom{s}{i}   \Bigg) \mymu[k]{i} 
\]

\[
\implies \sum\limits^{(m)} (s+1)^k = \sum\limits_{i=0}^{k}   \binom{s+m}{m+i}   \mymu[k]{i} 
\]

\[
\implies \sum\limits^{(m)} (s+1)^k = \sum\limits_{i=0}^{k}   \binom{(s+1)+m-1}{m+i}   \mymu[k]{i} 
\]

Finally, we proved the formula of $m-$nested summation of first $n$ natural numbers with power $k$.

\textbf{Note :} An alternative way to generate the entries from Saras triangle is by using the following formula:

\end{proof}

For $k \in \mathbb{W}$ and $r \in \mathbb{W}$
$$\mymu[k]{r} = \sum\limits_{i=0}^{r} \binom{r}{i} (-1)^i (r+1-i)^k$$

\begin{proof}
This statement can also be proved by using mathematical induction.

Let $P(n,r)$ be the given statement. That is

\[
P(n,r): \mymu[k]{r} = \sum\limits_{i=0}^{r} \binom{r}{i} (-1)^i (r+1-i)^k
\]

\underline{Base case :} We prove $P(0,r)$ as basis case

\[
 P(0,r) = \sum\limits_{i=0}^{r} \binom{r}{i} (-1)^i (r+1-i)^0 = \sum\limits_{i=0}^{r} \binom{r}{i} (-1)^i =      \begin{cases}
       1 &\quad\text{if   } r = 0 \\ 
       0 &\quad\text{if   } r > 0 \\
       \end{cases}
\]
\underline{Induction step}:

In this step we prove the statement $\Big[ P(s, r-1) \land P(s,r)\Big] \implies P(s+1, r)$

\[
\mymu[s+1]{r} = r (\mymu[s]{r-1}) +(r+1) (\mymu[s]{r}) = r \sum\limits_{i=0}^{r-1} \binom{r-1}{i} (-1)^i (r-i)^s + (r+1) \sum\limits_{i=0}^{r} \binom{r}{i} (-1)^i (r+1-i)^s
\]

\[
\implies \mymu[s+1]{r} = \sum\limits_{i=0}^{r-1} r \binom{r-1}{i} (-1)^i (r-i)^s +  \sum\limits_{i=0}^{r} (r+1) \binom{r}{i} (-1)^i (r+1-i)^s
\]

\[
\implies \mymu[s+1]{r} = \sum\limits_{i=1}^{r} r \binom{r-1}{i-1} (-1)^{(i-1)} (r+1-i)^s +  \sum\limits_{i=0}^{r} (r+1) \binom{r}{i} (-1)^i (r+1-i)^s
\]

\[
\implies \mymu[s+1]{r} =   \sum\limits_{i=0}^{r} (r+1) \binom{r}{i} (-1)^i (r+1-i)^s - \sum\limits_{i=1}^{r} r \binom{r-1}{i-1} (-1)^i (r+1-i)^s 
\]

\[
\implies \mymu[s+1]{r} =   \sum\limits_{i=0}^{r} (r+1) \binom{r}{i} (-1)^i (r+1-i)^s - \sum\limits_{i=1}^{r} i \binom{r}{i} (-1)^i (r+1-i)^s 
\]

\[
\implies \mymu[s+1]{r} =   \sum\limits_{i=0}^{r} (r+1) \binom{r}{i} (-1)^i (r+1-i)^s - \sum\limits_{i=0}^{r} i \binom{r}{i} (-1)^i (r+1-i)^s 
\]

\[
\implies \mymu[s+1]{r} =   \sum\limits_{i=0}^{r} \binom{r}{i} (-1)^i (r+1-i)^{(s+1)}
\]

Hence the statement is proved.
\end{proof}

\section*{Conclusion}

We presented a general framework for generating formulas for nested summation by introducing Saras triangle.

\medskip

\noindent MSC2010: 40B05, 40B99

\end{document}